\documentclass[12pt]{article}

\usepackage{amssymb}
\usepackage{amsthm}
\usepackage{amsfonts}
\usepackage[T1]{fontenc}
\usepackage{epsfig}
\usepackage[lflt]{floatflt}
\usepackage{epsfig}
\usepackage{graphicx}
\usepackage{amsmath}
\usepackage{amssymb}
\usepackage{color}
\usepackage{subfigure}

\usepackage[utf8]{inputenc}
\usepackage{amsthm}
\usepackage{amsfonts}
\usepackage[T1]{fontenc}
\usepackage{epsfig}
\usepackage[utf8]{inputenc}
\usepackage{hyperref}
\usepackage{wrapfig} 
\usepackage{soul}
\usepackage[lflt]{floatflt}
\usepackage{epsfig}
\usepackage{graphicx}
\usepackage{amsmath}
\usepackage{amssymb}
\usepackage{color}
\usepackage{subfigure}
\usepackage[english]{babel}
\usepackage{empheq}
\usepackage{caption}
\usepackage{stmaryrd}
\usepackage{todonotes}
\usepackage{comment}

\usepackage{makeidx}
\usepackage{latexsym}
\usepackage{enumerate}
\newtheorem{theo}{Theorem}
\newtheorem{defi}{Definition}
\newtheorem{prop}{Proposition}
\newtheorem{lemma}{Lemma}
\newtheorem{coro}{Corollary}
\newtheorem{remark}{Remark}

\newcommand{\la}{\lambda}

\newcommand{\rmd}{{\rm d}}

\newcommand{\vfi}{\varphi}

\newcommand{\wi}{W^\infty}
\newcommand{\binf}{b^\infty}

\newcommand{\LWi}{\low{W}_\infty}

\newcommand{\lbb}{\underline b}

\newcommand{\UWi}{\overline{W}^\infty}

\newcommand{\uW}{\overline W}
\newcommand{\lW}{\underline W}
\newcommand{\bR}{\mathbb{R}}
\newcommand{\ubb}{\overline{b}}

\definecolor{wineRed}{rgb}{0.7,0,0.3}

\newcommand{\lowW}{\underline{W}}
\newcommand{\lowb}{\underline{b}}
\newcommand{\upW}{\overline{W}}
\newcommand{\upb}{\overline{b}}
\newcommand{\low}[1]{\underline{#1}}

\begin{document}

\title{Convergence of solutions of a one-phase Stefan problem with  Neumann boundary data to a self-similar profile}
\author{Danielle  Hilhorst$^{1}$,  Sabrina  Roscani$^{2,3}$, Piotr Rybka$^4$ \\
\small{$^1$CNRS and Laboratoire de Math\'ematiques, University Paris-Saclay} \\
\small{91405 Orsay Cedex, France} \\
\small{$^{2}$ CONICET}\\ \small{$^3$ Depto. Matem\'atica, FCE, Univ. Austral}\\ \small{Paraguay 1950, Rosario, S2000FZF, Argentina}\\
$^4$ \small{Department of Mathematics, University of Warsaw},\\ 
\small{ul. Banacha 2 02-097 Warszawa, Poland}}
\date{}
\maketitle

\begin{abstract}
We study a one-dimensional one-phase Stefan problem with a Neumann boundary condition on the fixed part of the boundary. We construct the unique self-similar solution, and show that starting from arbitrary initial data, solution orbits converge to the self-similar solution.
\end{abstract}



\section{Introduction}\label{Sec:Introduction} 
\subsection{On the content of the paper}
The Stefan problem has been extensively studied in the past decades. Despite the number of articles and  books published on this topic, \cite{Ru1971}, \cite{Fr1988}, 
\cite{Me1992}, \cite{Gu2018}, there are still open problems left, see for instance \cite{aiki}, \cite{BoHi2023}. One of the questions which requires further attention is the long time behavior of the one-phase Stefan problem, where the heat flux is specified at the fixed boundary, namely the Neumann problem :
\begin{equation}\label{NSP}
\begin{array}{cll}
(i) & {u}_t(x,t)- u_{xx} (x,t) =0, & t > 0, \ 0<x<s(t),\\
(ii) &  -u_x(0,t)=\displaystyle{\frac{h}{\sqrt{t+1}}},\quad u(s(t),t)= 0, & t >0,\\
(iii)& \dot{s}(t)=-u_x(s(t),t), & t >0,\\
(iv) & u(x,0)=u_0(x), & 0<x<s(0)=b_0,
\end{array}
\end{equation}

We stress that this type of boundary condition is reasonable from the modeling view point. Namely, the decay of data presented in (\ref{NSP}-$ii$) is consistent with the parabolic scaling. We choose, however, to shift the initial time by a positive constant, say 1. In this way we avoid an artificial singularity at the initial time $t=0$. Let us note that the existence of a unique smooth solution $(u,s)$ to Problem (\ref{NSP}) has already been  established. We recall the assumptions of this result in Subsection 1.2. Here $u$ may be called the temperature and $s$ is the position of the interface.

Our approach to study the long time behavior of Problem (\ref{NSP}) follows a general heuristics saying that the time asymptotics is determined by the steady states (there is none for (\ref{NSP})) or special solutions such as self-similar solutions or travelling waves. In fact, we show in Corollary \ref{c-sss} that there is exactly one self-similar solution $(v,\sigma)$, which has the form, $v(x,t) = U\left( \frac x{\sqrt{t+1}} \right)$ and $\sigma(t) = \omega \sqrt{t+1}$ for some constant $\omega$ and a profile function $U.$ Our main result states that the self-similar solution is attracting. 
\begin{theo}\label{tw1}
Suppose that $(u_0, b_0)$   satisfies the conditions 
\begin{equation}\label{zalu0-D}
0\le u_0\in W^{1,\infty}(0,+\infty), \quad 0< u_0(0)
\ \text{ and }   u_0(x) =0 \,\text{in } [b_0,\infty).
\end{equation}
Let $(u,s)$ be the corresponding solution of Problem (\ref{NSP}). Then,\\
\,\rm{1)} $\lim\limits_{t\to\infty}s(t)/\sqrt{t+1} = \omega$;\\
\rm{2)} 
$\displaystyle{\lim\limits_{t\to\infty}\sup_{x/\sqrt{t+1}\in[0,\omega]} \left| u(x,t) - U\left(\frac x{\sqrt{t+1}}\right)\right|} =0.$
\end{theo}

Our method of proof is based upon recent results obtained by \cite{Bou}, \cite{BoHi2023}, who use the comparison principle in an essential way. The argument  dwells on the possibility of trapping a given solution to (\ref{NSP}) between two solutions with known time asymptotic behavior. In order to make this method work we transform (\ref{NSP}) to a problem on a bounded domain with the  help of similarity  variables. The self-similar solution of (\ref{NSP}) corresponds to the steady state solution of the transformed system. Its uniqueness is of crucial importance for the proof. 

Let us stress the main difference 
between \cite{Bou}, \cite{BoHi2023} and the present article. The authors of \cite{Bou} and \cite{BoHi2023} present a quite technical proof to show that the space derivative of the solution uniformly converges to its limit as $t\to\infty.$ Here, we completely avoid  such a claim, so that
our proof is simpler and more direct, which would make our method easier to adapt to a different setting.

We should point out that there are a number of results dealing with the asymptotic behavior of solutions of Stefan problems, mainly in the case of Dirichlet data on the fixed boundary. 
However, even for Dirichlet data, there are not so many articles besides \cite{Bou} and 
\cite{BoHi2023} simultaneously addressing the behavior of the temperature profile $u$ and the shape of the interface $s.$

\subsection{Existence and uniqueness of the solution}

Let us stress that (\ref{zalu0-D}) is our standing set of assumptions on the initial conditions. 
Moreover, the condition $u_0(0)>0$ is necessary to construct proper lower solutions, but is not needed in the Proposition below:
\begin{prop}
	Assume that the initial condition $u_0$ satisfies \eqref{zalu0-D} and  that $h>0$. Then, there exists a unique classical solution $(u,s)$ of Problem \eqref{NSP} for all $t>0$, in the following sense : 
	\begin{equation*}
			s \in {C^1}([0, \infty)),\quad  u \in C^{2,1}(\{(x,t): t>0, 0 < x < s(t)\}), 	\end{equation*} 
\begin{equation*}		u \in {C(\{(x,t): t \geq 0, 0 \leq x \leq s(t)\}}), u_x \in {C(\{(x,t): t>0, 0 \leq x \leq s(t)\})}.
			\end{equation*}
\end{prop}

We refer to \cite[Chapter 8, Theorem 2]{Fr1964} for the proof of this proposition. 

Strictly speaking, the original 
statement in \cite{Fr1964} required $u_0$ to be of class $C^1$; however, we may relax this assumption in view of \cite[Theorem 5.1]{An2004}.

\bigskip
The organization of this paper is as follows. In Subsection 2.1 we discuss the
existence and uniqueness of the self-similar solution. Subsection 2.2 is devoted to the study of upper and lower solutions as well as to estimates following from monotonicity. In the last Section, Section 3, we present the proof of the convergence result which is based on the comparison principle.

\section{Self-similar, lower and  upper solutions}
\subsection{Self-similar solution}
We start by re-expressing Problem (\ref{NSP}) in terms of the
self-similar variables. In other words, we set
$$
W(\eta,\tau)=u(x,t)\quad\hbox{and} 
\quad b(\tau)=\frac{s(t)}{\sqrt{t+1}},
$$ 
where $\displaystyle{\eta=\frac{x}{\sqrt{t+1}}}$ 
and $\tau=\ln(t+1)$,  to obtain the problem 
\begin{equation}\label{NSP-W}
\begin{array}{cll}
(i) & W_\tau(\eta,\tau)-W_{\eta\eta}(\eta,\tau)-\displaystyle{\frac{ \eta}{2}} W_\eta (\eta,\tau)=0 & \tau > 0,\ 0<\eta<b(\tau),\\
(ii) &  -W_\eta(0,\tau)= h,\quad  W(b(\tau),\tau)= 0
  &  \tau > 0,\\
(iii)& \dot{b}(\tau)+\displaystyle{\frac{b(\tau)}{2}}=-W_\eta(b(\tau),\tau)  &  \tau > 0,\\
(iv) &  b(0)=b_0>0, \quad W(\eta,0)=u_0(\eta)& 0<\eta< b_0.
\end{array}
\end{equation}
Let us remark that the existence and uniqueness  of the stationary solution of problem \eqref{NSP-W}  were given in \cite{Ta1981}. 

\begin{lemma}\label{self-sim-sol} The  associated stationary problem to  \eqref{NSP-W}, which 
 is given by 
\begin{equation}\label{ODE-P}
\begin{array}{ll}
(i)& W_{\eta \eta}(\eta)+\displaystyle{\frac{\eta}{2}} W_\eta(\eta)=0, \quad  0<\eta<\omega,\\
(ii)& -W_\eta(0)= h,\quad  W(\omega)=0,\\
(iii)&\displaystyle{\frac{\omega}{2}}=- W_\eta(\omega).
\end{array}
\end{equation}
admits a unique solution given by the pair $(U,\omega)$ such that 
 \begin{equation}\label{sss}
    U(\eta) =  h
 \displaystyle{\int_\eta^\omega e^{-\frac{s^2}{4}}\,ds} , \qquad \eta\in [0,\omega]
\end{equation}
and  $\omega$ is the unique positive solution of the equation
$h= \displaystyle{\frac{x}{2} e^{\frac{x^2}{4}}} .$
\end{lemma}
We immediately conclude from this result that
\begin{coro}\label{c-sss}
If we set $u(x,t) = U(\displaystyle{\frac x{\sqrt{t+1}}})$  and $\sigma = \omega \sqrt{t+1}$, then $(u,\sigma)$ is a solution to (\ref{NSP}$-i$)--(\ref{NSP}$-iii$).
\end{coro}

\subsection{ Lower and upper solutions}

Similarly as 
in \cite{BoHi2023} we define notions of lower and upper solutions of (\ref{NSP-W}).
\begin{defi}\label{low-sol} We say that  a pair 
of smooth functions $(\lowW,\lowb)$ (resp. $(\upW,\upb))$ is a lower (resp. upper) solution of Problem \eqref{NSP-W} if 
\begin{equation}\label{low-NSP-W}
\begin{array}{cll}
(i) & W_\tau(\eta,\tau)-W_{\eta\eta}(\eta,\tau)-\displaystyle{\frac{ \eta}{2}} W_\eta (\eta,\tau)\leq 0\quad (\text{resp.} \geq 0),  \, & \tau > 0, \ 0<\eta<b(\tau),\\
(ii) &  -W_\eta(0,\tau)\leq  
  h\quad (\text{resp.} \geq h), 
  &  \tau > 0,\\
(iii) & W(b(\tau),\tau)=  0 & \tau > 0,\\
(iv) & \dot{b}(\tau)+\displaystyle{\frac{b(\tau)}{2}}\leq    -W_\eta(b(\tau,\tau))\quad (\text{resp.} \geq -W_\eta(b(\tau,\tau)))\,  & \tau > 0,\\
(v) & b(0)\leq  b_0  \quad (\text{resp. }b(0)\geq b_0) & \\
(vi) & W(\eta,0)\leq  u_0(\eta)\quad (\text{resp. } W(\eta,0)\geq  u_0(\eta))   & 0<\eta<b(0).
\end{array}
\end{equation}
\end{defi}

The following comparison principle is a fundamental tool in our article. 

\begin{theo}\label{comparison} Let $(\lowW_1(\eta,\tau),\lowb_1(\tau))$ (respectively,  $(\upW_2(\eta,\tau),\upb_2(\tau))$) be the extensions by zero of the lower (respectively, upper solutions) of \eqref{NSP-W}
corresponding 
to the data $(h_1,{u_0}_1,{b_0}_1)$ (respectively, $(h_2,{u_0}_2,{b_0}_2)$). 
If  $h_1 \leq h_2$, ${u_0}_1\leq {u_0}_2$ and ${b_0}_1\leq {b_0}_2$, then $\lowb_1(\tau)\leq \upb_2(\tau)$ for every $\tau>0$ and $\lowW_1(\eta,\tau)\leq \upW_2(\eta,\tau)$ for every $\eta\geq 0$ and $\tau\geq 0$.
\end{theo}

\begin{proof}
The proof is rather similar to those presented by \cite[Lemma 2.2 and Remark 2.3]{DU2015} and  \cite[Lemma 3.5]{DU2010}. We omit it here. 
\end{proof}
In fact, we will
construct lower and upper solutions,  which are independent of time. For this purpose, we present the perturbed stationary problem 
\begin{equation}\label{W-lambda-PB}
\begin{array}{cll}
(i) & W_{\eta\eta}(\eta)+\frac12\lambda \eta
W_\eta (\eta)=0 & 0<\eta<b_\la, \\
(ii) &  -W_\eta(0)= \tilde{h}, 
  &  W(b_\la)= 0,\\
(iii) & \displaystyle{\frac{b_\la}{2}}=-W_\eta(b_\la).  & 
\end{array}
\end{equation}
whose solution is given by the pair $(U_\la,b_\la)$, where
$U_\lambda(\eta)= \tilde{h}  
\int_\eta^{b_\la} e^{-\lambda s^2/4}\rmd s, 
$
for given  $\lambda$, and  
$b_\la$ is the 
unique solution to
\begin{equation}\label{b_la} \tilde{h}
=\frac{b_\la}{2}\,e^{\lambda b_\la^2/4}.
\end{equation}

\begin{remark}\label{U-la-props}It is easy to see that for every  $\lambda>0$ and $\eta$ in $(0,b_\la)$,
$U_\lambda\geq 0$, $
(U_\lambda)_\eta < 0$ and 
$(U_\lambda)_{\eta\eta}> 0$. 
In particular, $U_\lambda$ is a linear function for $\lambda=0$, and it is a strictly convex function if $\lambda>0$.
\end{remark} 
\begin{lemma}\label{LowUp} Let $(U_\lambda,b_\lambda)$ be a solution of 
\eqref{W-lambda-PB}. 
If $\tilde{h}
\leq h$, then for all $\lambda > 1$  
$(U_\lambda,b_\lambda)$ is an independent of time lower solution to \eqref{NSP-W}. 
\end{lemma}
\begin{proof} 
First we show that if $\lambda > 1$, solutions of \eqref{W-lambda-PB} are lower solutions. Indeed 
\begin{equation}
- \{({U_\lambda})_{\eta\eta}+ \displaystyle{\frac{\eta}{2}} ({U_\lambda})_\eta \} = 
-\{ ({U_\lambda})_{\eta\eta}+\lambda \displaystyle{\frac{\eta}{2}} ({U_\lambda})_\eta\}  - \displaystyle{\frac{\eta - \lambda \eta}{2}} ({U_\lambda})_\eta = \displaystyle{\frac{\eta (\lambda - 1)}{2}} ({U_\lambda})_\eta < 0,
\end{equation}
for all $\eta \in (0, b_\lambda)$. Now, from \eqref{b_la} and   the inequality $e^x > x$ for all $x>0$,  it holds that
$0< b_\la=2 \tilde{h} e^{\la b^2_\la/4} \leq \displaystyle{\frac{8 \tilde{h}}{\la b^2_\la}}$, 
which implies that $b_\lambda\to0$ as $\la\to \infty$; thus we can choose $\la$ large enough so that $b_\la<b_0$. 

Next we show that we can choose $\lambda$ such that 
$U_\lambda \leq u_0$. On the one hand we have that  
$$ u_0(\eta)=u_0(0)+\int_0^\eta u_0'(s)\, d s \geq u_0(0)-M\eta, \quad \text{for all }\, 0\leq \eta\leq b_0. $$
On the other hand, for every $0<\eta<b_\la$, 
we deduce from the strict convexity of $U_\lambda$ discussed in Remark \ref{U-la-props} that
\begin{equation}\label{Ula-inic-cond} 
	U_\lambda(\eta) <  U_\lambda(0) +\frac{U_\lambda(b_\la)-U_\lambda(0)}{b_\la}\eta =U_\lambda(0) -\frac{U_\la(0)}{b_\la}\eta. \end{equation}
Also we remark that
\begin{equation}\label{U_la(0)to0}
U_\la(0)= \tilde{h}
\int_0^{b_\la}e^{-\la s^2}\rmd s\leq  \tilde{h} 
b_\la\to 0, \quad \la\to \infty,
\end{equation}
and recall that by the hypothesis \eqref{zalu0-D} $u_0(0) > 0$.
Thus, if we choose $\la$ large enough so that $b_\la<\displaystyle{\frac{u_0(0)}{M}}$ and $U_\la(0)<u_0(0),$ it follows that 
\begin{equation}\label{U_la<u0}  U_\lambda(\eta) \leq u_0(\eta) \quad \hbox{for all } \eta \in [0,b_\la]. \end{equation}
We conclude that $(U_\la,b_\la)$ is a lower solution according to Definition \ref{low-sol}.
\end{proof}

\begin{figure}[h]
\includegraphics[scale=0.8]{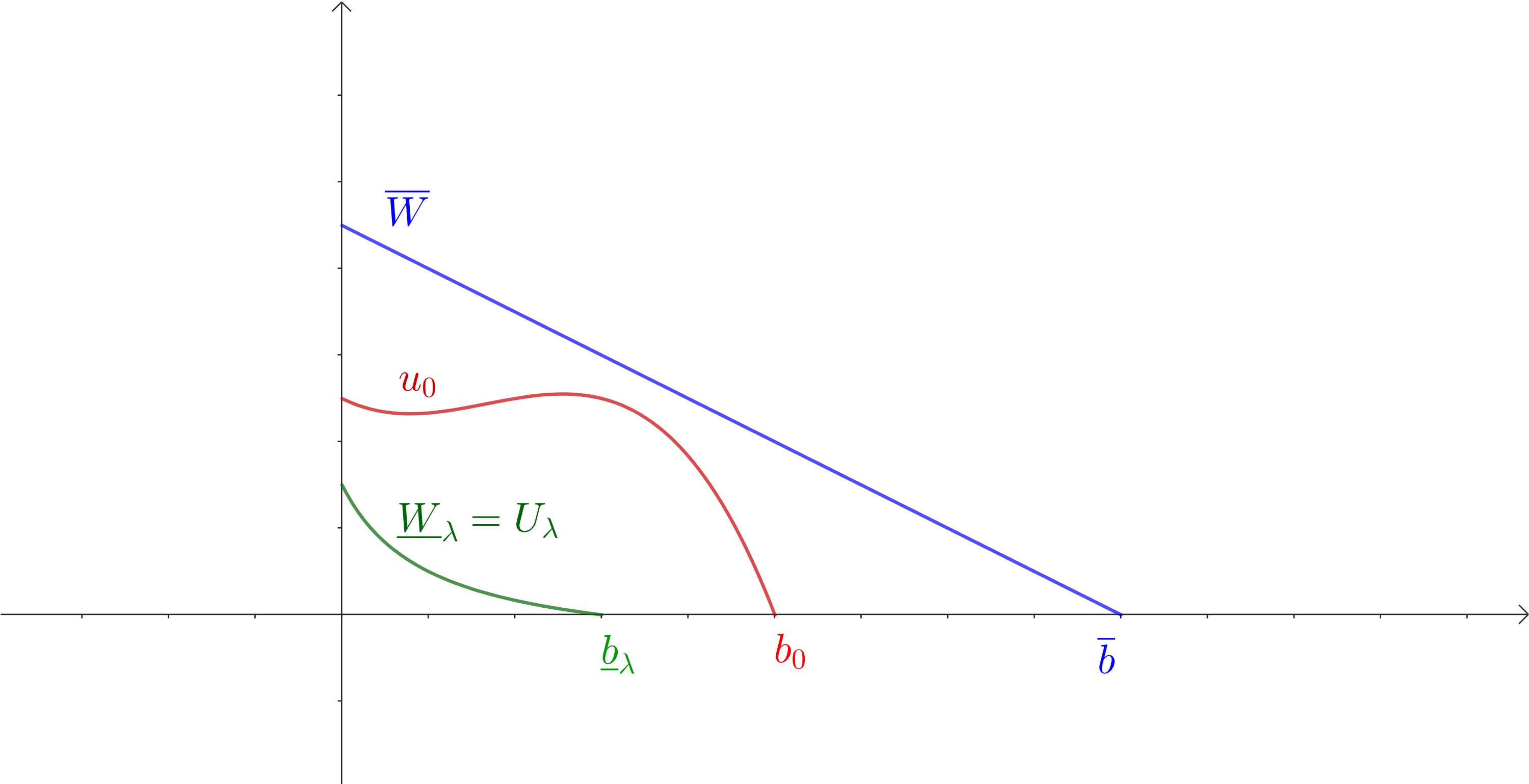}
\caption{Lower and upper solutions}
\end{figure}

Now, 
we define  the pair $(\lowW_\la,\lowb_\la)$ by 
\begin{equation}\label{lowSol}
\lowb_\la=b_\la \quad \text{and}\quad \lowW_\la (\eta):=\begin{cases}U_\la(\eta) & \text{ if } 0\leq \eta\leq \lowb_\la,\\ 0 &  \text{ if } \eta > \lowb_\la  \end{cases}.
\end{equation}
Next we propose an upper solution which is a straight line on its support. It is easy to verify that the pair $(\upW,\upb)$ defined by 
\begin{equation}\label{upSol}
\upb\geq b_0 \, \text{ with }\, \upb\geq 2h \quad \text{and}\quad \upW (\eta):=\begin{cases} \displaystyle{\frac{\upb}{2}}(\upb-\eta), & \text{ if } 0\leq \eta\leq \upb,\\ 0 &  \text{ if } \eta > \upb  \end{cases}
\end{equation}
is an upper solution. Next, we give an additional condition in order to ensure that $U_0 \geq u_0$ on the interval $[0, b_0]$. Since 
$$
u_0(\eta) \leq M (b_0 - \eta)  \mbox{ for all } \eta \in (0, b_0),
$$ 
we deduce that if $\upb  \geq \sqrt {2M b_0}$,
$$
U_0(\eta) \geq u_0 (\eta) \mbox{ for all } \eta \in (0,b_0),
$$
which implies a similar property for $\overline W$.

\section{Convergence}
In this section, the pair $(\lowW_\la,\lowb_\la)$ is the lower solution for a fixed $\la$ given by \eqref{lowSol} and $(\upW,\upb)$ is the upper solution given in \eqref{upSol}. 
We shall write
\begin{equation}\label{Sol-for-Low}
\low{W}(\eta,\tau):=W(\eta,\tau,(\lowW_\la,\lowb_\la)), \quad  \lowb(\tau)=b(\tau,(\lowW_\la,\lowb_\la))
\end{equation}
and 
\begin{equation}\label{Sol-for-Up}{\overline W}(\eta,\tau):=W(\eta,\tau,(\upW,\upb)), \quad  \upb(\tau)=b(\tau,(\upW,\upb))
\end{equation}
to denote solutions of (\ref{NSP-W}).

\begin{lemma}\label{Boundedness} \textsc{Positivity and boundedness}\\ 
There holds :
$$
0 \leq {\underline W}_\lambda(\eta)\leq {\underline W(\eta, \tau)} \leq W(\eta, \tau, u_0, b_0) \leq {\overline W(\eta, \tau)} \leq {\overline W}(\eta) \leq \displaystyle{\frac{\upb^2}{2}},
$$
and 
$$
0 \leq b_\lambda \leq b(\tau, u_0, b_0) \leq  \upb.
$$
\end{lemma}
\begin{proof}
Repeatedly apply the comparison principle Theorem \ref{comparison}.
\end{proof}
\begin{lemma}\cite{BoHi2023}\label{W-b-mono} \textsc{Monotonicity in time}\\ 
a) The functions $\low{W}(\eta,\tau)$ and $\lowb(\tau)$, are non-decreasing  in time.\\
b) The functions ${\overline W}(\eta,\tau)$ and $\upb(\tau)$, are non-increasing  in time.
\end{lemma}

\begin{proof} We only prove part $a)$. For the sake of simplicity of notation we suppress the lower bar and we write $W$. 
From Theorem \ref{comparison} we deduce that 
\begin{equation}\label{dec-1}
     W(\eta,s,(\lowW_\la,\lowb_\la))\geq \lowW_\la(\eta) \quad \text{and } \quad b(s,(\lowW_\la,\lowb_\la))\geq \lowb_\la, \quad \text{for all } \, s \geq 0.
\end{equation}
Now, for a fixed $s=\sigma$, we consider the pair $(W^\sigma,b^\sigma)$ where 
\begin{equation}\label{sigma-initial}
    W^\sigma(\eta)=W(\eta,\sigma,(\lowW_\la,\lowb_\la)) \quad \text{and }\quad b^{\sigma}=b(\sigma,(\lowW_\la,\lowb_\la)).
\end{equation}
In particular we have 
$W^\sigma(\eta)\geq \lowW_\la(\eta) \quad \text{and }\quad b^{\sigma}\geq \lowb_\la.$
Then we apply again Theorem \ref{comparison} to deduce that for every $ \tau \geq 0$
\begin{equation}\label{dec-3}
    W(\eta,\tau,(W^\sigma,b^\sigma)\geq W(\eta,\tau,(\lowW_\la,\lowb_\la))  \quad \text{and }\quad b(\tau,(W^\sigma,b^\sigma))\geq b(\tau,(\lowW_\la,\lowb_\la)).
\end{equation}
Returning to \eqref{dec-1}, now consider  $s=\tau+\sigma$ for $\tau\geq 0$. It holds that the pair 
 $(W(\eta,\tau+\sigma,(\lowW_\la,\lowb_\la)),b(\tau+\sigma,(\lowW_\la,\lowb_\la)))$ is a solution to problem \eqref{NSP-W} for the initial conditions \eqref{sigma-initial} for every $\tau > 0$. From the uniqueness of the solution we deduce that for all $\tau \ge 0$ we have
 \begin{equation}\label{dec-4}
    W(\eta,\tau,(W^\sigma,b^\sigma)= W(\eta,\tau+\sigma,(\lowW_\la,\lowb_\la))   \  \text{and } 
    b(\tau,(W^\sigma,b^\sigma))= b(\tau+\sigma,(\lowW_\la,\lowb_\la)). 
\end{equation}
Substituting  \eqref{dec-4} in \eqref{dec-3} we deduce that 
\begin{equation}\label{dec-5}
    W(\eta,\tau+\sigma,(\lowW_\la,\lowb_\la))\geq W(\eta,\tau,(\lowW_\la,\lowb_\la))  
    \text{ and } 
    b(\tau+\sigma,(\lowW_\la,\lowb_\la))\geq b(\tau,(\lowW_\la,\lowb_\la)),
\end{equation}
which completes the proof of part a).
\end{proof}
We remark that if $(\low{W},\lowb)$ (resp. $({\overline W}, \upb)$) is defined in (\ref{Sol-for-Low}) (resp. (\ref{Sol-for-Up})), then the Lemmas \ref{Boundedness} and  \ref{W-b-mono} imply that  for every $\lambda>0$
$$
0< b_\lambda\le \lim\limits_{\tau\rightarrow \infty}\lowb(\tau)=\lowb^\infty \le 
\lim\limits_{\tau\rightarrow \infty}\upb(\tau)=\upb^\infty \leq \upb.
$$
In addition, the Lemmas \ref{Boundedness} and  \ref{W-b-mono} imply the convergence of $\low{W}$ and  ${\overline W}$, namely 
\begin{equation}\label{psi}
0 \leq \lim\limits_{\tau\rightarrow \infty} \low{W}(\eta,\tau)=\low{W}^\infty(\eta) \leq \frac{\upb^2}{2}, \ 
\mbox{for all  }\, \eta \in [0,\lowb^\infty]; \
\end{equation}
and
\begin{equation}\label{phi}      
0 \leq \lim\limits_{\tau\rightarrow \infty} {\overline W}(\eta,\tau)= {\overline W}^\infty(\eta) \leq \frac{\upb^2}{2} \, \ 
 \mbox{for all  } \, \eta \in [0,\upb^\infty).
\end{equation} 
Finally we state and prove the main result of this paper, which in turn implies the result of Theorem 
\ref{tw1}.
\begin{theo}\label{t-con}
Let $(W(\eta,\tau,(u_0,b_0)), b(\tau,(u_0,b_0)))$ be the solution 
to the free boundary problem \eqref{NSP-W} 
associated to the initial data $(u_0,b_0)$.
If $(W^\infty, b^\infty)$ is  the unique steady state of \eqref{NSP-W} given by Lemma \ref{self-sim-sol}, then\\
\rm{(a)} $\displaystyle{\lim_{\tau\to\infty} b(\tau) = \binf}$,\\
\rm{(b)} $W(\cdot, \tau)$ converges to $\wi$ uniformly on $[0,\beta]$
for any $\beta \in (0, \binf)$, when $\tau\to \infty$.\\
\rm{(c)} At each time $\tau$ let us extend $W$ to $(0, \infty)^2$ by the formula
$$
\tilde{W}(\eta,\tau) = \left\{
\begin{array}{ll}
W(\eta,\tau)   &  \eta\in[0,b(\tau)],\\
    0 & \eta\in (b(\tau), \infty).
\end{array}
\right.
$$
Then, $\tilde{W}(\cdot, \tau)$ converges uniformly to $\tilde{W}^\infty$,
where
$$
\tilde{W}^\infty(\eta) = \left\{
\begin{array}{ll}
W^\infty(\eta),   &  \eta\in[0,b^\infty],\\
    0, & \eta\in (b^\infty, \infty).
\end{array}
\right.
$$

\end{theo}
\begin{proof}.
{\it  Step 1.} We have to identify the limits $\LWi$ and $\UWi$, and improve the convergence. For this purpose, we shall show the estimate,
\begin{equation}\label{rs0.1}
 \int_T^{T+1}\int_0^{b(\tau)} W_\eta^2(\eta,\tau)\, d\eta d\tau \le M_1,
\end{equation}
where  $W= \low{W}$ or $W={\overline W}$ are the functions defined in \eqref{Sol-for-Low} and \eqref{Sol-for-Up}, respectively.

In principle, we do not know if $W_{\eta\eta}$ and $W_\tau$ are square integrable over $\Omega_{T,1} :=\{(\eta,\tau): \eta\in(0, b(\tau)), \tau\in(T, T+1)\}$. This is why we set $\psi^\delta\in W^{1,\infty}(\{(\eta,\tau): \eta\in(0, b(\tau)), \tau\in(0,  \infty)\}) $
$$
\psi^\delta(\eta, \tau) = \left\{
\begin{array}{ll}
0 & \eta \in [0, \delta) ,\\
\frac 1\delta (\eta - \delta) &  \eta \in [ \delta, 2 \delta),\\
1 & \eta \in [2 \delta, b(\tau) - 2\delta),\\
1-\frac 1\delta (\eta -b(\tau) + 2\delta) &  \eta \in [ b(\tau) - 2\delta, b(\tau) - \delta),\\
0 &\eta\in [b(\tau) - \delta,\infty)
\end{array}
\right.
$$
and $\phi^\delta \in W^{1,\infty}((0,  \infty))$, where 
$$
\phi^\delta(\tau) = \left\{
\begin{array}{ll}
    0 &  \tau \in [0, T+\delta), \\
    \frac1\delta (\tau -T-\delta) & \tau \in [T+\delta, T+2\delta),\\
   1  &   \tau \in [T+2\delta, T+1-2\delta)), \\
    1-\frac1\delta (\tau -(T+1-2\delta)) & \tau \in [T+1-2\delta, T+1-\delta),\\
    0 &  \tau \in [T+1-\delta, \infty).
\end{array}
\right.
$$
Finally, we set $\vfi^\delta (\eta,\tau)= \psi^\delta(\eta,\tau)\phi^\delta(\tau)$. Now, let us multiply the equation $(\ref{NSP-W})-(i)$ by $\vfi^\delta W$ and  integrate over $\Omega_{T,1}$. We arrive at
$$
L_1^\delta(T) := \int_{\Omega_{T,1}}W W_\tau \vfi^\delta\, d\eta d\tau 
=\int_{\Omega_{T,1}}( W W_{\eta\eta} + \frac\eta2 W_\eta W) \vfi^\delta\, d\eta d\tau =: R_1^\delta(T).
$$
We first analyze the left-hand-side, we see that integration by parts yields,
\begin{equation*}
2L_1^\delta(T) = 
\int_{\Omega_{T,1}} \vfi^\delta (W^2)_\tau \,  d\tau d\eta  = 
-\int_{\Omega_{T,1}} \vfi^\delta_\tau \ W^2\, d\tau d\eta.
\end{equation*}
The boundary terms vanish, because the support of $\vfi^\delta$ does not intersect $\partial\Omega_{T,1}.$ Now, we want to  compute the limit $L_1(T) =\lim\limits_{\delta\to 0^+} L_1^\delta(T) $. We remark that
\begin{align*}
  L_1(T) 
  =\lim_{\delta\to 0^+}\left(
  - \frac1{2\delta} \int_{T}^{T+1}\int_{b(\tau)-2\delta}^{b(\tau)-\delta}\dot{b}(\tau)
  \phi^\delta(\tau) W^2(\eta,\tau)  \,  d\eta d\tau + \right.\\
  \left. \frac1{2\delta} \int_{T+1-2\delta}^{T+1-\delta}\int_0^{b(\tau)}
  \psi^\delta W^2  \,  d\eta d\tau - \frac1{2\delta}
\int_{T+\delta}^{T+2\delta}\int_0^{b(\tau)}
 \psi^\delta W^2  \,  d\eta d\tau\right). 
\end{align*}
For the limit in the first integral, note that from the continuity of the integrand, we can apply the mean value property for integrals to obtain that 
$$\lim_{\delta\to 0^+}
\frac1{2\delta} \int_{T}^{T+1}\int_{b(\tau)-2\delta}^{b(\tau)-\delta}\dot{b}(\tau)
\phi^\delta(\tau) W^2(\eta,\tau)  \,  d\eta d\tau= 
\frac{1}{2} \int_{T}^{T+1} \dot{b}(\tau)
W^2(b(\tau),\tau)  \,  d\tau. $$
Applying the condition at the moving boundary, we deduce that 
$$\lim_{\delta\to 0^+}
   \frac1{2\delta} \int_{T}^{T+1}\int_{b(\tau)-2\delta}^{b(\tau)-\delta}\dot{b}(\tau)
  \phi^\delta(\tau) W^2(\eta,\tau)  \,  d\eta d\tau=0.$$
As for the other terms, we proceed in a similar way to deduce that
 $$\lim_{\delta\to 0^+} \frac1{2\delta} \int_{T+1-2\delta}^{T+1-\delta}\int_0^{b(\tau)}
\psi^\delta W^2  \,  d\eta d\tau=  \frac12
\int_0^{b(T+1)} W^2 (\eta, T+1) \,  d\eta, $$
and that
 $$\lim_{\delta\to 0^+} \frac1{2\delta} \int_{T+\delta}^{T+2 \delta}\int_0^{b(\tau)}
 \psi^\delta W^2  \,  d\eta d\tau=  \frac12
\int_0^{b(T)} W^2 (\eta, T) \,  d\eta,.$$
Hence, we conclude that
$$
L_1(T) =  \frac12
\int_0^{b(T+1)} W^2 (\eta, T+1) \,  d\eta - \frac12
\int_0^{b(T)} W^2 (\eta, T)\,  d\eta.
$$
Next we consider the term $R^\delta_1(T)$. We integrate by parts to obtain
$$
R^\delta_1(T) =-\int_{\Omega_{T,1}}
( W_\eta^2 \vfi^\delta+  W W_{\eta} \vfi^\delta_\eta +
\frac 14  W^2 (\eta \vfi^\delta_\eta + \vfi^\delta))\, d\eta d\tau. 
$$
Again here, the boundary terms vanish, because  the support of $\vfi^\delta$ is contained in $\Omega_{T,1}$. 
We remark that the same argument as above leads us to
\begin{align*}
    \lim_{\delta\to 0^+} \int_{\Omega_{T,1}} W W_{\eta} \vfi^\delta_\eta\, d\eta d\tau &=
    \lim_{\delta\to 0^+} \frac1\delta
    \int_{T}^{T+1} \left( \int_\delta^{2\delta}
    W W_{\eta} \phi^\delta\, d\eta d\tau 
    -  \int_{b(\tau) -2\delta}^{b(\tau)-\delta} W W_{\eta} \phi^\delta \, d\eta d\tau\right)
    \\ &=
\int_T^{T+1} (-h  W(0,\tau)- W(b(\tau), \tau) W_\eta(b(\tau), \tau)) \, d\eta d\tau \\
&= - \int_T^{T+1} h  W(0,\tau) \, d\eta d\tau,
\end{align*}
where we have also used the boundary condition (\ref{NSP-W}$-ii$) as well as the condition at the interface (\ref{NSP-W}$-iii$).
Similarly one can show that
\begin{align*}
    \lim_{\delta\to 0^+} \int_{\Omega_{T,1}} \frac 14  W^2 (\eta \vfi^\delta_\eta + \vfi^\delta))\, d\eta d\tau =
    \frac14\int_{\Omega_{T,1}} W^2(\eta,\tau)\,d\eta d\tau.
\end{align*}
Hence, we conclude
$$
R_1(T) = \lim_{\delta\to 0^+} R^\delta_1(T) =
-\int_{\Omega_{T,1}} W_\eta^2 \,  d\tau d\eta+ h \int_T^{T+1} W(0,\tau)\,  d\tau -\frac14 \int_{\Omega_{T,1}} W^2 \,  d\tau d\eta.
$$

In view of Lemma   \ref{Boundedness}  a possible choice of the constant $M_1$ is given by              
$$
M_1 = h \frac{\upb^2}{2} +  \frac{\upb^5}{8}.
$$
{\it Step 2.} Applying the mean value theorem for integrals in  (\ref{rs0.1}) we deduce from Step 1 that there exist two sequences of  points  $\tau_n', \tau_n''\in [n, n+1)$ such that 
\begin{equation}\label{rs4}
 \int_0^{{\underline b}(\tau_n')} \lW_\eta^2(\eta, \tau_n')\, d\eta,\quad \int_0^{{\overline b}(\tau_n'')} \uW_\eta^2(\eta, \tau_n'')\, d\eta \le M_1.
\end{equation}
Since $\lbb(\tau_n')\le {\underline b}^\infty \le  {\overline b}^\infty \le \ubb(\tau_n'')$, it follows that (\ref{rs4}) implies the bounds
\begin{equation}\label{bound}
\| \lW_\eta (\cdot, \tau_n')\|_{L^2(0,\beta)}, \ \| \uW_\eta (\cdot, \tau_n'')\|_{L^2(0, {\overline b}^\infty)}\le \sqrt{M_1}
\end{equation}
for all $\beta\in (0,{\underline b}^\infty)$. 
Hence, we can select  subsequences (not relabelled)  such that
$$
\lW_\eta(\cdot, \tau_n')\rightharpoonup \psi^{\beta} \hbox{ in } L^2(0,\beta) \quad\hbox{and}\quad \uW_\eta (\cdot, \tau_n'')\rightharpoonup \Psi\hbox{ in } L^2(0,\overline b^\infty).
$$
Since the limit is unique, we deduce that $\psi^{\beta}$ and $\Psi$ do not depend on the choice of the sequence $\tau_n$ and using \eqref{psi} and \eqref{phi} we conclude that  $\psi^{\beta}$ is the weak derivative of $\underline{ W}_\eta^\infty$ in every interval $(0,\beta)$ and that  $\Psi =  \overline W_\eta^\infty$. So that $\underline W^\infty {\chi_{(0,\beta)}}\in H^1(0,\beta)$
 for all $\beta \in (0, {\underline b}^\infty)$  and thus $\overline W^\infty \in H^1(0, \upb^\infty)$. In fact,  $\underline W^\infty \in H^1(0, {\underline b}^\infty)$. Indeed, in view of Lebesgue monotone convergence theorem we have
$$
\lim_{\beta \rightarrow {\underline b}^\infty}\int_0^{\underline b^\infty} ({\underline W}_\eta^\infty)^2(\eta) \chi_{(0,\beta)}(\eta)\,d\eta = \int_0^{{\underline b}^\infty}({\underline  W}_\eta^\infty)^2 (\eta)\,d\eta \le M_1.
$$
{\it Step 3.}  We claim that  $(\underline{W}^\infty, {\underline b}^\infty)$ and $(\UWi, {\overline b}^\infty)$ are both  stationary solutions, namely solutions  of \eqref{ODE-P}. Hence they are smooth and equal. Indeed, we multiply  the equation (\ref{NSP-W}-$i$) by $\vfi\in C^\infty_c(\bR)$ such that $\vfi_\eta(0)=0$ and we
integrate 
on $\Omega_{T,1}$. Proceeding as in step 2 we set
$$
L_2(T) = \int_{\Omega_{T,1}} W_\tau \varphi  \, d\eta d\tau 
=\int_{\Omega_{T,1}}(W_{\eta\eta} \varphi + \frac\eta2 W_\eta \varphi) \, d\eta d\tau = R_2(T).
$$
\begin{align*}
    L_2(T) &=\int_{\Omega_{T,1}} W_\tau (\eta,\tau)\vfi(\eta) \, d\eta d\tau \\&= \int_0^{b(T+1)} W(\eta, T+1) \vfi(\eta)\, d\eta -  \int_0^{b(T)} W(\eta, T) \vfi(\eta)\, d\eta,
\end{align*}
where $W= \lW$ or $W = \uW$. We then deduce from Lebesgue's dominated convergence theorem  that $\lim_{T\to\infty} L_2(T) =0.$

Next, we  investigate the right-hand-side $R_2(T)$. Integration by parts yields
\begin{align*}
R_2(T) & =
\int_{\Omega_{T,1}} W_{\eta\eta}\vfi\, d\eta  d\tau +
\int_{\Omega_{T,1}} \frac\eta2 W_\eta \vfi \, d\eta  d\tau \\
& = 
\int_{\Omega_{T,1}} W( \vfi_{\eta\eta}- \frac12(\eta\vfi)_{\eta}) \, d\eta d\tau 
- \int_T^{T+1} (\dot b + \frac b 2) \vfi(b(\tau)) +  h \vfi(0).
\end{align*}
Now, we pass to the limit as $T \rightarrow \infty$.  It follows from Lebesgue's dominated convergence theorem that 
$$
\lim_{T\to\infty}\int_{\Omega_{T,1}} W(\eta, \tau)( \vfi_{\eta\eta}- \frac12(\eta\vfi)_{\eta}) \, d\eta d\tau 
= \int_0^{b^\infty} W^\infty(\eta)( \vfi_{\eta\eta}- \frac12(\eta\vfi)_{\eta}) \, d\eta,
$$
where $(W^\infty, b^\infty)$ is either $(\underline{W}^\infty,{\underline b}^\infty)$ or $(\overline{W}^\infty, {\overline b}^\infty)$.
Let us  denote by $\Phi$ an antiderivative of $\vfi$. Then,
$$
\lim_{T\to\infty}\int_T^{T+1} \dot b \vfi(b(\tau))\, d\tau
= \lim_{T\to\infty} (\Phi(b(T+1)) - \Phi(b(T))) =0.
$$
In addition,
$$
\lim_{T\to\infty}\int_T^{T+1}\frac b 2 \vfi(b(\tau))\, d\tau
= \frac12 \binf \vfi(\binf).
$$
Finally, we  collect all the results concerning $R_2(T)$, while keeping in mind that $\lim_{T\to\infty}L_2(T)=0$. This yields
\begin{equation}\label{r8}
0=\lim_{T\to\infty} R_2(T) =
\int_0^{\binf} \wi (\eta)( \vfi_{\eta\eta}- \frac12(\eta\vfi)_{\eta}) \, d\eta
- \frac\binf 2 \vfi(\binf) + h \vfi(0),
\end{equation}
for all smooth functions $\varphi$ in $\bR$ such that $\vfi_\eta(0)=0$.
In particular $\wi$ satisfies the differential equation (\ref{ODE-P}-$i$) in the sense of distributions.\\

\noindent {\it Step 4.} We recall that $\wi \in H^1(0,\binf)$. It is easy to infer from (\ref{r8}) that  $\wi_{\eta\eta} = - \displaystyle{\frac\eta2} 
W^\infty_\eta\in L^2(0, \binf)$, which in turn implies that  $W^\infty \in H^2(0,b^\infty)$. \\

Next we search for the boundary condition and the conditions on the moving boundary satisfied by $\wi$. 
After integrating by parts twice in  (\ref{r8}) we obtain,
\begin{align*}
 0&=
\int_0^{\binf}(\wi_{\eta\eta} +\frac\eta2\wi_\eta)\vfi\,d\eta
+\wi \vfi_\eta|_{\eta=0}^{\eta=\binf}
-\wi_\eta \vfi |_{\eta=0}^{\eta=\binf} \\
&- \frac\binf 2 \vfi(\binf)(\wi(\binf) +1) + h \vfi(0),
\end{align*}
so that
\begin{align}\label{r3}
0=& \wi(\binf)\varphi_\eta (\binf)-\wi_\eta(\binf)\vfi(\binf) +\wi_\eta(0) \vfi(0)\nonumber \\
&- {\frac{\binf}{2}}  \vfi(\binf)(\wi(\binf) +1) + h \vfi(0),
\end{align}
for all smooth functions $\varphi$ on $\bR$ such that $\vfi_\eta(0)=0$.
Now, if we additionally choose $\vfi$ such that $\vfi(\binf) = \vfi_\eta (\binf)=0$, then (\ref{r3}) reduces to 
$$
\vfi(0)( \wi_\eta(0) + h) =0,
$$
and since $\vfi(0)$ is arbitrary, we deduce that 
$$
\wi_\eta(0) + h =0.
$$
Thus (\ref{r3}) becomes
\begin{equation}
\label{r4}
		0= \wi(\binf)\varphi_\eta (\binf) -\wi_\eta(\binf)\vfi(\binf) 
		- {\frac{\binf}{2}}  \vfi(\binf)(\wi(\binf) +1).
\end{equation}
Next we suppose that $\vfi(\binf) =0$, but $\vfi_\eta(\binf) \neq 0$. Then 
$$
 \wi(\binf) \vfi_\eta(\binf) =0,
$$
and hence 
$$
 \wi(\binf) =0.
 $$ 
 Then, (\ref{r4}) becomes
 \begin{equation}\label{r5}
 0= -\wi_\eta(\binf)\vfi(\binf)  - {\frac{\binf}{2}}  \vfi(\binf).
 \end{equation}
Suppose that $\vfi(\binf) \neq 0$. Then  (\ref{r5}) implies that
$$
\wi_\eta(\binf) = - \frac\binf 2.
$$
We deduce that the solution pair $(W^\infty, b^\infty)$ coincides with the unique solution of  Problem (\ref{ODE-P}) or in other words with the unique steady state solution of the time evolution problem, Problem (\ref{NSP-W}).\\

\noindent {\it Step 5.} We recall that, in view of step 4, $\wi\in H^2(0,\binf)\subset C^{1, \frac12}([0,\binf])$. Moreover,
the convergence of $\uW$ to $\wi$ (resp. ${\tilde{\lW}}$ to $W^\infty$) is monotone on $[0,\binf]$.  
Hence, we deduce with the help of Dini's Theorem that this convergence is uniform.

Finally, since $b_\lambda \leq b_0 \leq \upb$ and since ${\underline w}_\lambda \le u_0\le {\overline W}$, the comparison principle implies that ${\underline b}(\tau) \leq b(\tau) \leq {\overline b}(\tau)$ and ${\tilde{\underline W} (\eta,\tau)\le 
\tilde{W}(\eta,\tau)}\le 
W (\eta,\tau)\le \uW (\eta,\tau)$ for all $(\eta, \tau)$. We conclude that $b(\tau) \rightarrow b^\infty$  and that $\tilde W(\tau)$ converges to $\wi$ uniformly on compact sets of $[0, \infty)$ as $\tau \rightarrow \infty$. 
\end{proof}

Now, Theorem \ref{tw1} easily follows from Theorem \ref{t-con}.

\subsection*{Acknowledgments}
The second author, SR, was partly supported by the Project INV-006-00030 from Universidad Austral
and PICT-2021-I-INVI-00317. The work of the second  author, PR, was in part supported by the National Science Centre, Poland, through the grant number
2017/26/M/ST1/00700. A part of the work was performed while DH and SR visited the University of Warsaw, whose hospitality is greatly appreciated.


\begin{thebibliography}{00}

\bibitem{aiki}
T. Aiki, A. Muntean,  A free-boundary problem for concrete carbonation: Front nucleation and rigorous justification of the  $\sqrt t$-law of propagation, {\it Interfaces Free Bound.}, {\bf 15} (2012), 167-180.
   
\bibitem{An2004} {D. Andreucci}, {\it Lecture Notes on the Stefan Problem}, {Universita di Roma La Sapienza} (2004).

\bibitem{BoHi2023}
M. Bouguezzi, D. Hilhorst, Y. Miyamoto and J. F. Scheid, Convergence to a self-similar solution for a one-phase Stefan problem arising in corrosion theory, {\it European J. Appl. Math.} {\bf 34} (2023), 701--737.

\bibitem{Bou}
M. Bouguezzi,
{\it Modeling and computer simulation of the propagation rate of pit corrosion}, Ph.D thesis Université Paris-Saclay  (2021).

\bibitem{DU2010}
Y. Du and Z.  Lin, Spreading-vanishing dichotomy in the diffusive logistic model with a free boundary, 
{\it SIAM J. Math. Anal.}, {\bf 42}  (2010), 377--405.

\bibitem{DU2015}
Y. Du and B. Lou, Spreading and vanishing in nonlinear diffusion problems with free boundaries, {\it J. Eur. Math. Soc.} {\bf 17} (2015), 2673--2724.


\bibitem{Fr1964} 
A. Friedman, {\it Partial Differential Equations of Parabolic Type},  
(Prentice-Hall, 1964).
   
 \bibitem{Fr1988}  
 A. Friedman, {\it Variational Principles and Free Boundary Broblems}, 
 (Robert E. Krieger Publishing Co. Inc.,1988).
 
  \bibitem{Gu2018}
 S. C. Gupta, {\it The Classical Stefan Problem. Basic Concepts, Modelling and Analysis},  
 (Elsevier, 2018).
 
  \bibitem{Me1992}
 A. Meirmanov, {\it The Stefan Problem}, 
 (Walter de Gruyter, 1992).

 \bibitem{Ru1971} 
 L.I. Rubinstein, {\it The Stefan Problem, Translations of Mathematical Monographs} {\bf 27}, (American Mathematical Society, 1971). 


\bibitem{Ta1981}
 Tarzia, D. A., 
 An inequality for the coeficient $\sigma$ of the free boundary $s(t)=2\sigma \sqrt{t}$ of the {N}eumann solution for the two-phase {S}tefan problem, {\it Quart. Appl. Math.} {\bf 39} (1981), 491--497. 


\end{thebibliography}
\end{document}